\documentclass[12pt]{amsart}
\usepackage{graphicx}
\usepackage[T1]{fontenc}
\usepackage{color}
\usepackage{tikz}
\usepackage{hyperref}

\definecolor{uniblau}{cmyk}{1,0.5,0,0}

\newtheorem{theorem}{Theorem}[section]
\newtheorem{lemma}[theorem]{Lemma}
\newtheorem{prop}[theorem]{Proposition}
\newtheorem{cor}[theorem]{Corollary}

\theoremstyle{definition}
\newtheorem{defi}[theorem]{Definition}

\theoremstyle{remark}
\newtheorem{remark}[theorem]{Remark}

\numberwithin{equation}{section}

\newcommand{\R}{{\mathbb R}}
\newcommand{\nat}{{\mathbb N}}

\newcommand{\C}{{\mathbb C}}
\newcommand{\tim}{{\mathbb T}}

\newcommand{\E}{\mathbb E}

\allowdisplaybreaks

\begin{document}

\sloppy

\title[Integral representation of dilatively stable processes]{An integral representation of dilatively stable processes with independent increments}

\author{Thorsten Bhatti}
\address{Thorsten Bhatti, Mathematical Institute, Heinrich-Heine-University D\"usseldorf, Universit\"atsstr.\ 1, D-40225 D\"usseldorf, Germany}
\email{thorsten.bhatti\@@{}uni-duesseldorf.de}

\author{Peter Kern}
\address{Peter Kern, Mathematical Institute, Heinrich-Heine-University D\"usseldorf, Universit\"atsstr.\ 1, D-40225 D\"usseldorf, Germany}
\email{kern\@@{}hhu.de}

\thanks{This work has been supported by the German Academic Exchange Service (DAAD) financed by funds of the Federal Ministry of Education and Research (BMBF) under project no. 57059326}

\date{\today}

\begin{abstract}
Dilative stability generalizes the property of selfsimilarity for infinitely divisible stochastic processes by introducing an additional scaling in the convolution exponent. Inspired by results of Igl\'oi \cite{Igloi}, we will show how dilatively stable processes with independent increments can be represented by integrals with respect to time-changed L\'evy processes. Via a Lamperti-type transformation these representations are shown to be closely connected to translatively stable processes of Ornstein-Uhlenbeck-type, where translative stability generalizes the notion of stationarity. The presented results complement corresponding representations for selfsimilar processes with independent increments known from the literature.
\end{abstract}



\maketitle

\baselineskip=18pt

\section{Introduction}

{ Many processes in physics and other sciences show certain space-time scaling properties for which the class of self-similar processes provides a natural tool in stochastic modeling. For infinitely divisible processes Igl\'oi \cite{Igloi} introduced a more general scaling property called dilative stability with an additional scaling in the convolution exponent. We denote by $\Psi^{X}_{t_1,\ldots,t_k}$ the {\it log-characteristic function} or the {\it L\'evy exponent} of $(X_{t_1},\ldots,X_{t_k})$ for an infinitely divisible process $X=(X_t)_{t\in\mathbb T}$, where $\mathbb T$ is either $\R$ or $\R_+ = [0, \infty)$ and $t_1,\ldots,t_k\in\mathbb T$, i.e.\ $\Psi^{X}_{t_1,\ldots,t_k}:\R^k\to\C$ is the unique continuous function with $\Psi^{X}_{t_1,\ldots,t_k}(0,\ldots,0)=0$ and
$$\E \left[\exp\Big( i \sum_{j=1}^k \theta_jX_{t_j}\Big)\right] = \exp\left( \Psi^{X}_{t_1,\ldots,t_k}(\theta_1,\ldots,\theta_k)\right)$$
for all $\theta_1,\ldots,\theta_k\in\R$. Following \cite{BPK}, the infinitely divisible process $X$ is called $(\alpha,\delta)$-{\it dilatively stable} for some parameters $\alpha,\delta\in\R$ if 
\begin{equation}\label{DSdefi}
\Psi^{X}_{Tt_1,\ldots,Tt_k}(\theta_1,\ldots,\theta_k)=T^{\delta}\Psi^{X}_{t_1,\ldots,t_k}(T^{\alpha-\frac{\delta}{2}}\theta_1,\ldots,T^{\alpha-\frac{\delta}{2}}\theta_k)
\end{equation}
holds for all $T>0,\,k \in \mathbb N,\,t_1,\ldots,t_k\in\mathbb T$ and $\theta_1,\ldots,\theta_k\in\R$. It is immediately clear that for $\delta=0$ and $\alpha>0$ an $(\alpha,\delta)$-dilatively stable process is $\alpha$-selfsimilar. We remark that the original definition of Igl\'oi \cite{Igloi} is more restrictive (e.g., the process is assumed to be non-Gaussian and to possess moments of arbitrary order) but we use the more general approach from \cite{BPK}. The class of dilatively stable processes contains some interesting classes of processes that are not selfsimilar, see \cite{Igloi,BPK} for details. In particular, additionally assuming weak right-continuity of the infinitely divisible process $X$, dilative stability of $X$ is equivalent to the notion of aggregate-similarity introduced by Kaj \cite{Kaj}, see Proposition 1.5 in \cite{BPK}. From this point of view, dilatively stable processes naturally appear as the class of limit processes in certain aggregation models as shown in Theorem 3.1 of \cite{KW}. Examples of dilatively stable limit processes in aggregation schemes appear in \cite{Kaj,PS}, see section 3 in \cite{BPK} for a detailed analysis.}

In this paper we will restrict our considerations to {\it additive processes} $(X_t)_{t\in\mathbb T}$ which are defined as in \cite{Sat1} by the following conditions:
\begin{itemize}
\item[(i)] The process has {\it independent increments}, i.e.\ for any $ t_0 < t_1 < \ldots <t_n$ in $\mathbb T$ the random variables $X_{t_0}, X_{t_1}- X_{t_0}, X_{t_2}-X_{t_1}, \ldots, X_{t_n} - X_{t_{n-1}}$ are independent.
\item[(ii)] The process is {\it stochastically continuous}, i.e.\ $P\{ | X_s -X_t | > \varepsilon\}\to0$ as $s\to t\in\mathbb T$ for any $\varepsilon>0$. 
\item[(iii)] The process has {\it c\`adl\`ag paths}, i.e.\ almost surely the mapping $t \mapsto X_t$ is right-continuous with left limits.
\item[(iv)] $X_0 = 0$ almost surely.
\end{itemize}
It is well known that additive selfsimilar processes are closely connected to selfdecomposable random variables and thus can be represented as integrals with respect to a L\'evy process; cf.\ Wolfe \cite{Wol1,Wol2}, Jurek and Vervaat \cite{Jur,JurVer} and Sato \cite{Sat0}. In order that the random integrals do properly exist, the L\'evy process necessarily must have a finite logarithmic moment. Certain extensions of the integral representation for additive operator-selfsimilar and semi-selfsimilar processes are given in \cite{JPY} and \cite{BK,MaeSat2}, respectively. Further, the Lamperti transform \cite{Lam} gives a well known correspondence between selfsimilar processes and stationary processes. The latter are stationary Ornstein-Uhlenbeck (OU) processes in case of additive selfsimilar processes and the integral representation of an additive selfsimilar process is directly related to the integral representation of the corresponding OU-process.

Our aim is to generalize the above mentioned integral representations and connections for the larger class of additive dilatively stable processes in Section 2. As already laid out in sections 2.5 and 2.6 of Igl\'oi \cite{Igloi}, these are integrals with respect to certain time-changed L\'evy processes. Since Igl\'oi \cite{Igloi} requires finite moments of arbitrary order, for our less restrictive definition \eqref{DSdefi} of dilative stability we are particularly interested in deriving appropriate moment conditions for the driving L\'evy processes. In case $\alpha=\delta/2$ it will turn out that there is also a connection to additive quasi-selfsimilar processes introduced by Maejima and Ueda \cite{MU2}; see Remark \ref{MaeUeda}.

A certain generalization of the Lamperti transform directly relates dilatively stable processes to so-called translatively stable processes. The definition of the latter processes also goes back to Igl\'oi \cite{Igloi} and is directly connected to processes which are infinitely divisible in time and to time-stable processes as introduced by Mansuy \cite{Man}, respectively Kopp and Molchanov \cite{KM}. We will lay out this connection in Section 3 and, inspired by Igl\'oi \cite{Igloi}, we will further show how our integral representation of additive dilatively stable processes from Section 2 is related to certain translatively stable processes of OU-type via a Lamperti-type transformation.

\section{Random Integral representation}

We use the following construction of random integrals as almost surely pathwise limits of Riemann-Stieltjes sums which goes back to Wolfe \cite{Wol1,Wol2} or Jurek and Vervaat \cite{Jur,JurVer}, cf. also Lemma 2.1 in \cite{BK}.
Let $Y=(Y_t)_{t \in {\mathbb T}}$ be an additive process on $\R$ and let $A:{\mathbb T} \to\R$ be continuously differentiable. Then for any $ a < b < \infty $ and any sequence of partitions  $a=t_0^{(n)} \leq s_1^{(n)} < t_1^{(n)} \leq s_2^{(n)} < \ldots \leq s_n^{(n)} < t_n^{(n)} =b$ of $[a,b] \subseteq {\mathbb T}$ with $\max_{1 \leq j \leq n}(t_j^{(n)} - t_{j-1}^{(n)}) \to 0$ as $ n \to \infty $ we have 
 \begin{align*}
 \sum_{j=1}^n \big( A(t_j^{(n)}) - A(t_{j-1}^{(n)}) \big) \, Y_{s_j^{(n)}} \to \int_a^b A'(t) Y_t \, dt \quad\text{ almost surely,}
 \end{align*}
 where the integral exists pathwise as a Riemann integral and the
 exceptional nullset does not depend on the particular choice of partitions. Now we are able to define a {\it random integral} by formal integration by parts
 \begin{equation}\label{ibp}
  \int_a^b A(t) dY_t := A(b)Y_b - A(a)Y_a - \int_a^b A'(t)Y_t \,dt
 \end{equation}
and this random integral can be pathwise approximated by Riemann-Stieltjes sums 
\begin{align*}
 \sum_{j=0}^n A(t_j^{(n)})\big(Y_{s_{j+1}^{(n)}} - Y_{s_j^{(n)}}\big) \to
 \int_a^b A(t) dY_t \quad\text{ almost surely,}
\end{align*}
where we define $s_0^{(n)} := a $ and $s_{n+1}^{(n)}:= b$. In this context the additive process $Y$ is called the {\it background driving process}. We will frequently make use of the following change of variables formula which is an easy consequence of the random integral construction.
For a continuous non-increasing or non-decreasing function $\gamma : {\mathbb T} \to \mathbb R$ and $[a,b] \subseteq {\mathbb T}$ we have 
 \begin{align*}
  \int_a^b A(\gamma(t)) d\big(Y_{\gamma (t)} - Y_{\gamma(0)}\big) = \int_{\gamma
  (a)}^{\gamma(b)} A(t) dY_t,
 \end{align*}
where in case $\gamma(b) < \gamma(a)$ the random integral on the right-hand side is defined by 
\begin{align*}
 \int_c^d A(t)dY_t = - \int_d^c A(t) dY_t \quad\text{ for any } c,d \in\tim.
\end{align*}
Our background driving process $Y$ itself will be defined by random integrals of the following kind.
\begin{lemma}\label{driving}
Let $X=(X_t)_{t \geq 0}$ be an additive process. Then the process
$Y=(Y_t)_{t \in \mathbb R}$ { given by}
\begin{equation}\label{Ydef}
Y_t =\int_{1}^{e^t} u^{-\alpha + \delta /2 } dX_{u}= \begin{cases}
 \int_{1}^{e^t} u^{-\alpha + \delta /2 } dX_{u}  , & \text{for } t \geq 0 \\
-\int_{e^t}^{1} u^{-\alpha + \delta /2 } dX_{u}, & \text{for }t < 0
\end{cases}
\end{equation}
is again additive.
\end{lemma}
\begin{proof}
By the definition of the random integral,
$(Y_t)_{t \in \mathbb R}$ has independent increments because $(X_t)_{t \geq 0}$ has and its paths are almost surely c\`adl\`ag functions. Of course $Y(0)=0$. It remains to check that $(Y_t)_{t \in
\mathbb R}$ is stochastically continuous. For $s,t \in \R$ we have 
$${ Y_t-Y_s} =e^{-t (\alpha - \delta/2 )} X_{e^t} -e^{-s (\alpha - \delta/2 )}
X_{e^s} + \int_{e^s}^{e^t} (\alpha - \delta/2) u^{-(\alpha - \delta/2 )-1} X_u
du,$$ where $e^{-t (\alpha - \delta/2 )} X_{e^t} -e^{-s (\alpha - \delta/2 )}
X_{e^s} \to 0$ in probability as $s \to t$, since $(X_t)_{\geq 0}$ is stochastically continuous,
and the integral converges to zero almost surely as $s\to t$, since the integrand is almost surely bounded on compact sets; see \cite[p.114]{JurMas}. 
\end{proof}
This directly enables us to obtain a random integral representation for additive dilatively stable processes.
\begin{theorem}\label{rirep}
Let $(X_t)_{t\geq0}$ be an additive $(\alpha,\delta)$-dilatively stable process and $(Y_t)_{t \in \mathbb R}$ be the corresponding process given in Lemma \ref{driving}. Then for any $t>0$ we have
$$X_t = \int_{-\infty}^{\log t} e^{u (\alpha-\delta / 2)}\, dY_u\quad\text{ almost surely.}$$
\end{theorem}
\begin{proof}
We will first prove a corresponding representation for the increments. By the construction of the random integral we obtain for $0<s<t$ 
$$ \int_{ \log s}^{\log t} e^{u(\alpha - \delta / 2)}\,dY_u=t^{\alpha - \delta /2} Y_{\log t} - s^{\alpha - \delta /2
} Y_{\log s} - \int_{\log s}^{\log t} (\alpha - \tfrac{\delta}{2}) e^{u(\alpha - \delta/ 2)}\,Y_u\, du.$$
For the latter integral we get by \eqref{Ydef}
\begin{align*}
& \int_{\log s}^{\log t} (\alpha - \tfrac{\delta}{2}) e^{u(\alpha - \delta/ 2)}\,Y_u\, du
= \int_{\log s}^{\log t} (\alpha - \tfrac{\delta}{2}) e^{u(\alpha - \delta/ 2)}\int_1^{e^u}v^{-\alpha+\delta/2}\,dX_v\, du\\
& \quad=\int_1^s v^{- \alpha + \delta /2} \int_{\log s}^{\log t}(\alpha - \tfrac{\delta}{2})e^{u(\alpha - \delta / 2)}\,du\,dX_v\\
& \qquad+\int_s^t v^{- \alpha + \delta /2} \int_{\log v}^{\log t}(\alpha - \tfrac{\delta}{2})e^{u(\alpha - \delta / 2)}\,du\,dX_v \\
& \quad= \int_1^s v^{- \alpha + \delta /2}\left(t^{\alpha - \delta /2}-s^{\alpha - \delta /2}\right)\,dX_v
+\int_s^t v^{- \alpha + \delta /2}\left(t^{\alpha - \delta /2}-v^{\alpha - \delta /2}\right)\,dX_v\\ 
& \quad=\left(t^{\alpha - \delta /2}-s^{\alpha - \delta /2}\right) Y_{\log s}+t^{\alpha - \delta /2}\left(Y_{\log t}-Y_{\log s}\right)-\int_s^tdX_v\\
& \quad=t^{\alpha - \delta /2} Y_{\log t} - s^{\alpha - \delta /2
} Y_{\log s}-(X_t-X_s),
\end{align*}
where exchangeability of the order of integration follows by \eqref{ibp} and Fubini's theorem. Together it follows that for $0<s<t$ we have
$$X_t-X_s = \int_{\log s}^{\log t} e^{u (\alpha-\delta / 2)}\, dY_u.$$
By stochastic continuity $X_t-X_{s}\to X_t-X_0=X_t$ in probability as $s\downarrow0$. 
Since the process $X$ has independent increments, convergence in probability is equivalent to almost sure convergence by Theorem A2.2 in \cite{JurVer}.
\end{proof}
We will now show that the backround driving process $(Y_t)_{t\in\R}$ is a time-transformed L\'evy process. Preparatory, we will investigate its increments.
\begin{lemma}\label{incre}
For fixed $T \in \R$ let $ (K_{t;T})_{t \in \R}:=(Y_{t+T} - Y_T)_{t \in \R}$, 
where $Y$ is the process from Lemma \ref{driving} with an additive $(\alpha,\delta)$-dilatively stable process $X$. In terms of the L\'evy exponent this process fulfills
\begin{equation}\label{LeK}
  \Psi_{t_1,\ldots, t_k;T}^K(\theta_1,\ldots,\theta_k) = e^{\delta T}
  \Psi_{t_1,\ldots, t_k}^Y(\theta_1,\ldots,\theta_k)
\end{equation}
for all $k\in\nat$ and $t_1,\ldots,t_k,\theta_1,\ldots,\theta_k\in\R$, where $\Psi_{t_1,\ldots, t_k;T}^K$ denotes the log-characteristic function of $(K_{t_1;T},\ldots,K_{t_k;T})$.
\end{lemma}
\begin{proof}
Let us first prove that $ \Psi_{t;T}^K(\theta) = e^{\delta T}\Psi_{t}^Y(\theta)$ for $t,\theta\in\R$.  { In case $t=0$ there is nothing to prove.} By definition of $ K_{t;T}$ we get
$$K_{t;T} =  Y_{t+T} - Y_{T} = \int_{e^T}^{e^{t+T}} u^{-\alpha + \delta /2 } dX_{u} = \int_{1}^{e^{t } }
 (ue^T)^{-\alpha + \delta /2 } dX_{ue^T}.$$
In case $t>0$ this gives us the following approximation by Riemann-Stieltjes sums for a sequence of
partitions $ 1 =s_0^{(n)}= t_0^{(n)} \leq s_1^{(n)} < t_1^{(n)} \leq \ldots \leq s_n^{(n)} < t_n^{(n)}=s_{n+1}^{(n)} = e^t$ 
\begin{align*}
&K_{t;T} =\int_1^{e^t} (ue^T)^{-\alpha + \delta /2 } dX_{ue^T}= \lim_{n \to \infty} \sum_{j=0}^{n} (t_j^{(n)} e^T)^{-\alpha + \delta /2 }
\left( X_{s_{j+1}^{(n)} e^T} - X_{s_{j}^{(n)} e^T} \right).
\end{align*}
To derive the Fourier transforms we use the following property
\begin{align*}
\widehat{P}_{X_{t_1}-X_{t_2}}(\theta) & = \mathbb E\big[ e^{i\theta (X_{t_1} - X_{t_2}  )} \big] 
=\int_{\R^2}  e^{i\theta (x_{1} - x_{2}  )} dP_{(X_{t_1},X_{t_2})}(x_1,x_2) \\
& = \widehat{P}_{(X_{t_1},X_{t_2})}(\theta ,-\theta ) = \exp\big(\Psi_{t_1,t_2}^X(\theta ,-\theta )\big).
\end{align*}
for all $t_1,t_2,\theta\in\R$. By L\'evy's continuity theorem we obtain for $\theta\in\R$
\begin{align*}
\exp\big(\Psi_{t;T}^K(\theta)\big) & = \widehat P_{K_{t;T}}(\theta) \\
&=\lim_{n\to\infty}\prod_{j=0}^n \exp \left( \Psi^X_{s_{j+1}^{(n)}e^T,s_{j}^{(n)}e^T} \left(
(t_j^{(n)} e^T)^{-\alpha + \delta /2 }\theta  , - (t_j^{(n)} e^T)^{-\alpha + \delta /2
} \theta  \right) \right) \\
&=\lim_{n\to\infty}\exp \left( \sum_{j=0}^{n} e^{\delta T} \Psi^X_{s_{j+1}^{(n)},s_{j}^{(n)}}
\left( (t_j^{(n)})^{-\alpha + \delta /2 }\theta  , - (t_j^{(n)})^{-\alpha + \delta /2} \theta 
\right) \right) \\
& = \left(\lim_{n\to\infty}\prod_{j=0}^{n}  \exp \left(  \Psi^X_{s_{j+1}^{(n)},s_{j}^{(n)}}
\left( (t_j^{(n)})^{-\alpha + \delta /2 }\theta  , - (t_j^{(n)})^{-\alpha + \delta /2} \theta 
\right) \right) \right)^{e^{\delta T}} \\
& =\left(\widehat P_{Y_{t}}(\theta )\right)^{e^{\delta T}} 
= \exp(e^{\delta T}\Psi_{t}^Y(\theta )),
\end{align*}
where in the third line we used the scaling property \eqref{DSdefi} of the dilatively stable
processes $X$. In conclusion $\Psi_{t;T}^K(\theta) = e^{\delta T}\Psi_{t}^Y(\theta)$ for $t>0$ and $\theta\in\R$. The same property holds for $t<0$ and $\theta\in\R$, since
$$K_{t;T}=Y_{t+T}-Y_T=-(Y_{(t+T)-t}-Y_{t+T})=-K_{-t;t+T}$$
and hence by the above we get
$$\Psi_{t;T}^K(\theta)=\Psi_{-t;t+T}^K(-\theta)=e^{\delta(t+T)}\Psi_{-t}^Y(-\theta)=e^{\delta T}\Psi_{-t;t}^K(-\theta)=e^{\delta T}\Psi_{t}^Y(\theta),$$
where the last equality follows by $K_{-t;t}=Y_0-Y_t=-Y_t$. Finally, we will prove \eqref{LeK} for $k=2$, the general case $k\in\nat$ follows inductively. Without loss of generality let $t_1<t_2$ then for $\theta_1,\theta_2\in\R$ by independence of the increments we have
\begin{align*}
\exp\left(\Psi^K_{t_1,t_2;T}(\theta_1,\theta_2)\right) &  = \E \big[\exp (i \theta_1 K_{t_1;T} + i \theta_2 K_{t_2;T} ) \big] \\ 
 & = \exp \left( \Psi^{K}_{t_1;T}(\theta_1 +\theta_2) \right) \exp \left( \Psi^{K}_{t_2-t_1;t_1+T}(\theta_2)\right) \\ 
 & = \exp \left( e^{\delta T} \Psi^{Y}_{t_1}(\theta_1 +\theta_2) \right) \exp \left( e^{\delta(t_1+ T)}  \Psi^{Y}_{t_2-t_1}(\theta_2)\right) \\
 & = \exp \left( e^{\delta T} \Psi^{Y}_{t_1}(\theta_1 +\theta_2) \right) \exp \left( e^{\delta T}  \Psi^{K}_{t_2-t_1;t_1}(\theta_2)\right) \\
 &=  \left( \E \big[\exp (i (\theta_1+\theta_2) Y_{t_1} )\big] \E \big[\exp (i \theta_2 (Y_{t_2}-Y_{t_1} ) )\big] \right)^{e^{\delta T}} \\ 
 &= \left( \E \big[\exp (i \theta_1 Y_{t_1} + i \theta_2 Y_{t_2} ) \big]\right)^{e^{\delta T}} = \exp\left(e^{\delta T}\Psi^Y_{t_1,t_2}(\theta_1,\theta_2)\right)\end{align*}
concluding the proof.
\end{proof}
Let $Y$ be an infinitely divisible random variable. We say that $(L(t))_{t\in\R}$ is the {\it two-sided L\'evy process} generated by the law of $Y$ if it can be represented as 
$$L(t) =  \begin{cases}
 L^{(1)}(t)   & \text{ if } t \geq 0 \\ 
  - L^{(2)}((-t)-) & \text{ if } t < 0 
 \end{cases} $$
with independent copies $(L^{(1)}(t))_{t\geq0},\,(L^{(2)}(t))_{t\geq0}$ of the L\'evy process generated by the law of $Y$. Note that $(L(t))_{t\in\R}$ has c\`adl\`ag paths.
\begin{lemma}\label{ttLP}
Let $X=(X_t)_{t \geq 0}$ be an additive $(\alpha,\delta)$-dilatively stable process. Then the background driving process $Y=(Y_t)_{t \in \R}$ from Lemma \ref{driving} is the time-changed process 
$$(Y_t)_{t\in\R} \stackrel{\rm d}{=} \big(L(\tfrac{e^{\delta t}-1}{e^{ \delta} - 1})\big)_{t\in\R},$$
where $(L(t))_{t\in\R}$ is the {\it two-sided L\'evy process} generated by the law of $Y_1$ and $\stackrel{\rm d}{=}$ denotes equality in distribution. Note that for $\delta = 0 $ the time-change function is simply defined by $\lim_{\delta \to 0} \big(
(e^{\delta t}-1)/(e^{\delta} - 1)\big)= t$. 
\end{lemma}
\begin{proof}
In case $\delta=0$ the process $X$ is a selfsimilar additive process and it is well known that the corresponding background driving process is a L\'evy process. Thus we will only prove the case $\delta\not=0$. For $N\in\nat$ and $n=0,\ldots, N-1$, setting $T=nt$ in Lemma \ref{incre} we get for any $t\in\R$
\begin{align*}
  Y_{Nt} = \sum_{n=0}^{N-1} Y_{t+nt} - Y_{nt} &= \sum_{n=0}^{N-1} K_{t;nt}.
\end{align*}
which by independence of the increments implies for the L\'evy exponents
$$\Psi_{Nt}^Y(\theta) = \sum_{n=0}^{N-1} e^{\delta n t} \Psi_t^Y(\theta) = 
\frac{e^{\delta N t} - 1 }{e^{\delta t } -1 } \Psi_t^Y(\theta)\quad\text{ for any }\theta\in\R.$$
Setting $t=1/N$ it follows that
$\Psi_{1}^Y(\theta) = \frac{e^{\delta } - 1 }{e^{\delta/N} -1} \Psi_{1/N}^Y(\theta)$
for any $N\in\nat$ and setting $t=1/m$ with $m\in\nat$ we get
$$\Psi_{N/m}^Y(\theta) = \frac{e^{\delta N/m} -1 }{e^{\delta/m} - 1 }  \Psi_{1/m}^Y(\theta)
 = \frac{e^{\delta N/m} -1 }{e^{\delta/m } - 1 } \cdot \frac{e^{\delta/m } -1 }{e^{\delta } - 1 } \Psi_{1}^Y(\theta) =\frac{e^{\delta N/m } -1 }{e^{\delta } - 1 }\Psi_{1}^Y(\theta).$$
Due to the stochastic continuity of $(Y_{t})_{t\in \mathbb R }$ we get 
\begin{equation}\label{bdp1}
Y_t\stackrel{\rm d}{=}L^{(1)}(\tfrac{e^{\delta t } -1 }{e^{\delta } - 1 })\quad\text{ for any }t\geq0.
\end{equation}
The same arguments for $t=-1/m$ with $m\in\nat$ together with Lemma \ref{incre} implies
\begin{align*}
\Psi_{-N/m}^Y(\theta)& = \frac{e^{-\delta N/m} -1 }{e^{-\delta/m} - 1 }  \Psi_{-1/m}^Y(\theta)
 = \frac{e^{-\delta N/m} -1 }{e^{-\delta/m } - 1 } \cdot \frac{e^{-\delta/m } -1 }{e^{-\delta } - 1 } \Psi_{-1}^Y(\theta)\\
& =-\frac{e^{-\delta N/m } -1 }{e^{\delta } - 1 }\,e^{\delta}\Psi_{-1}^Y(\theta)=-\frac{e^{-\delta N/m } -1 }{e^{\delta } - 1 }\Psi_{-1;1}^K(\theta)\\
& =-\frac{e^{-\delta N/m } -1 }{e^{\delta } - 1 }\Psi_{1}^Y(-\theta),
\end{align*}
where the last equality follows since $K_{-1;1}=Y_0-Y_1=-Y_1$. Hence, again due to the stochastic continuity of $(Y_{t})_{t\in \mathbb R }$ we get 
\begin{equation}\label{bdp2}
Y_t\stackrel{\rm d}{=}-L^{(2)}(-\tfrac{e^{\delta t } -1 }{e^{\delta } - 1 })\stackrel{\rm d}{=}-L^{(2)}((-\tfrac{e^{\delta t } -1 }{e^{\delta } - 1 })-)\quad\text{ for any }t<0.
\end{equation}
Combining \eqref{bdp1} and \eqref{bdp2} we get
\begin{equation}\label{bdp3}
Y_t\stackrel{\rm d}{=}L(\tfrac{e^{\delta t } -1 }{e^{\delta } - 1 })\quad\text{ for any }t\in\R
\end{equation}
and it remains to show that 
\begin{align*}
\widehat P_{\left(Y_{t_1},\ldots,Y_{t_k}\right)}(\theta_1,\ldots,\theta_k) = \widehat
P_{\big(L(\frac{e^{\delta t_1}-1}{e^{\delta}-1}) ,\ldots,L(\frac{e^{\delta
t_k}-1}{e^{\delta}-1})\big)}(\theta_1,\ldots,\theta_k)
\end{align*}
for all $k\in\nat$ and $t_1,\ldots,t_k,\theta_1,\ldots,\theta_k \in \R$. It suffices to prove the assertion for $k=2$ and $t_1<t_2$, the general case follows analogously. By independence of the increments of $(Y_t)_{t\in\R}$, Lemma \ref{incre} and \eqref{bdp3} we get
 \begin{align*}
\widehat{P}_{(Y_{t_1}, Y_{t_2})}(\theta_1,\theta_2) & = \widehat{P}_{(Y_{t_1}, Y_{t_2}-Y_{t_1})}(\theta_1+\theta_2,\theta_2)\\ 
& = \widehat{P}_{Y_{t_1}}(\theta_1+\theta_2)\cdot\widehat{P}_{Y_{t_2}-Y_{t_1}}(\theta_2)\\
& =\exp\left(\Psi^Y_{t_1}(\theta_1+\theta_2)\right)\cdot\exp\left(\Psi^K_{t_2-t_1;t_1}(\theta_2)\right)\\
& =\exp\left(\Psi^Y_{t_1}(\theta_1+\theta_2)\right)\cdot\exp\left(e^{\delta t_1}\Psi^Y_{t_2-t_1}(\theta_2)\right)\\
& =\widehat{P}_{L\left(\frac{e^{\delta t_1}-1}{e^{ \delta} - 1}\right)}(\theta_1+\theta_2) \cdot\widehat{P}_{L\big(e^{\delta t_1}\frac{e^{\delta(t_2-t_1)}-1}{e^{\delta} - 1}\big)}(\theta_2).
 \end{align*}
 Since the two-sided L\'evy process $(L(t))_{t\in\R}$ has stationary and independent increments we further get
 \begin{align*}
 \widehat{P}_{(Y_{t_1}, Y_{t_2})}(\theta_1,\theta_2)  & =\widehat{P}_{L( \frac{e^{\delta t_1}-1}{e^{ \delta} - 1})}(\theta_1+\theta_2) \cdot\widehat{P}_{L(\frac{e^{\delta t_2}-1}{e^{ \delta} - 1})-L( \frac{e^{\delta t_1}-1}{e^{ \delta} - 1})}(\theta_2)\\
 &=\widehat{P}_{\big(L(\frac{e^{\delta t_1}-1}{e^{ \delta} - 1}), L(\frac{e^{\delta t_2}-1}{e^{ \delta} - 1})\big)}(\theta_1,\theta_2)
\end{align*}
concluding the proof.
\end{proof}
Combining Theorem \ref{rirep} and Lemma \ref{ttLP} we immediately get the following representation in law.
\begin{cor}\label{rirepd}
Let $(X_t)_{t\geq0}$ be an additive $(\alpha,\delta)$-dilatively stable process and $(L(t))_{t \in \mathbb R}$ be the two-sided L\'evy process generated by the law of $Y_1=\int_1^e u^{-\alpha+\delta/2}\,dX_u$ from Lemma \ref{driving}. Then we have
$$(X_t)_{t\geq0} \stackrel{\rm d}{=} \Big(\int_{-\infty}^{\log t} e^{u (\alpha-\delta / 2)}\, dL( \tfrac{e^{\delta u}-1}{e^{ \delta} - 1})\Big)_{t\geq0}.$$
\end{cor}
\begin{remark}\label{MaeUeda}
In case $\alpha=\delta/2$ our representation in Theorem \ref{rirep} is trivial and we can only deduce that $(X_t)_{t\geq0}=(Y_{\log t}-Y_{-\infty})_{t\geq0}$, where $Y_{-\infty}$ exists almost surely due to stochastic continuity. By Corollary \ref{rirepd} we further get 
$$(X_t)_{t\geq0}\stackrel{\rm d}{=}(L(\tfrac{t^{\delta}-1}{e^\delta-1})-L(\tfrac{-1}{e^\delta-1}))_{t\geq0}\stackrel{\rm d}{=}L(\tfrac{t^{\delta}}{e^\delta-1})_{t\geq0}\quad\text{ provided }\delta>0.$$
Nevertheless, Maejima and Ueda \cite{MU2} provide an integral representation in case $\alpha=\delta/2$ as follows. Writing $\alpha=-\gamma/2$ and hence $\delta=-\gamma$ for $\gamma\in\R$, an additive $(-\gamma/2,-\gamma)$-dilatively stable process $(X_t)_{t\geq0}$ is $\gamma${\it -quasi-selfsimilar}, which by Definition 1.3 in \cite{MU2} means that the L\'evy exponent fulfills
$$\Psi_{Tt}^X(\theta)=T^{-\gamma}\Psi_t^X(T\theta)\quad\text{ for all }t\geq0,\,T>0,\,\theta\in\R.$$
This has the following interesting consequence. If $\gamma<0$ then by Theorem 2.2(II)(i) in \cite{MU2} the {\it Lamperti-type transform} $(Z_t=(1-\gamma t)^{1/\gamma}X_{(1-\gamma t)^{-1/\gamma}})_{t\geq0}$ is a $\gamma$-{\it mild OU-type process}, which by Definition 1.2(i) in \cite{MU2} means that
$$Z_t=(1-\gamma t)^{1/\gamma}\int_{1/\gamma}^t(1-\gamma u)^{-1/\gamma}\,Y(du),$$
where $Y$ denotes an independently scattered random measure. In conclusion we get
$$X_t=\int_{1/\gamma}^{\frac1\gamma(1+t^{-\gamma})}(1-\gamma u)^{-1/\gamma}\,Y(du)=\int_{-\infty}^{\log t}e^u\,Y(d\varphi(u)),$$
where the last equality applies by formal change of variables $\varphi(\log t)=\tfrac1\gamma(1+t^{-\gamma})$. Similar representations hold for $0<\gamma<2$ by Theorem 2.2(II) together with Definition 1.2 in \cite{MU2} if certain moment conditions on $Y$ are fulfilled. Note that by Theorem 2.3(II) in \cite{MU2} a further connection between $\gamma$-quasi-selfsimilar processes and $\gamma$-selfdecomposable random variables is outlined. The latter fulfill an integral representation by \cite{MU1}.

For $\alpha\not=\delta/2$ we will investigate a Lamperti-type transformation of dilatively stable processes and its connection to OU-type processes in Section 3.
\end{remark}
We now turn to the converse relation of constructing additive dilatively stable processes as random integrals with respect to a time-changed L\'evy process. As mentioned in the Introduction, for additive selfsimilar processes (case $\delta=0$) the driving L\'evy process must have a finite logarithmic moment. Since the desired converse relation for additive selfsimilar processes is already fully established in the mathematical literature, we concentrate on the case $\delta\not=0$. For $\delta<0$ we will need the following moment condition for which we were not able to find a suitable reference.
\begin{lemma}\label{sconv}
 Let  $(X_n)_{n \in \mathbb N}$ be an i.i.d.\ sequence. Then for $a,b \in \mathbb N$ with $a \geq 2$ and $\beta > 1$ we have
$$\sum_{k=0}^{\infty} a^{-k \beta} \sum_{\ell= 1}^{a^kb} X_\ell\quad\text{ converges absolutely almost surely iff }\quad
\E \left[| X_{1} |^{1/\beta}\right] < \infty.$$
\end{lemma}
\begin{proof}
Provided that the series converges absolutely almost surely, we may change the order of summation to get
\begin{equation}\label{umord}\begin{split}
\sum_{k=0}^{\infty} a^{-k \beta} \sum_{\ell= 1}^{a^kb} X_\ell & = \sum_{\ell=1}^{\infty} \Bigg( \sum_{k= \max\left\{\left\lceil \frac{\log(\ell/b)}{\log a}\right\rceil, 0 \right\}}^\infty a^{-k \beta}\Bigg)  X_\ell\\
&=\frac{1}{1-a^{-\beta}} \left( \sum_{\ell=1}^{b-1} X_\ell + \sum_{\ell=b}^{\infty}  a^{- \beta\left\lceil \frac{\log(\ell/b)}{\log a}\right\rceil}  X_\ell\right).
\end{split}\end{equation}
The latter series has independent summands and hence by Kolmogorov's three-series theorem we get for any $d>0$
\begin{align*}
\sum_{\ell=b}^{\infty} P \left\{ | X_1 |^{1/{\beta}} >  d^{1/{\beta}} \tfrac{a}{b}\, \ell\right\} & =\sum_{\ell=b}^{\infty}  P \left\{ | X_1 |^{1/\beta} >  d^{1/{\beta}}   a^{ \frac{\log(\ell/b)}{\log a} +1 } \right\} \\
& \leq\sum_{\ell=b}^{\infty} P \left\{ \left|a^{-\beta \left \lceil \frac{\log(\ell/b)}{\log a}\right \rceil } X_\ell \right| > d \right \}<\infty.
\end{align*}
Choosing $d=(\tfrac{b}{a})^\beta$ this shows that $\E \big[  |X_1|^{1/{\beta}} \big] < \infty$. \\ 
  Conversely, if $\E \big[  |X_1|^{1/{\beta}} \big] $ exists then 
  $\sum_{\ell=b}^{\infty} \ell^{-\beta} |X_\ell|$ converges almost surely, cf.\ Remark 3 in \cite{Chandra}. Thus we get 
   \begin{equation*}
  \sum_{\ell=b}^{\infty}a^{- \beta  \left\lceil \frac{\log(\ell/b)}{\log a } \right\rceil}  |X_\ell| \leq
   \sum_{\ell=b}^{\infty} \ell^{-\beta} |X_\ell|<\infty\quad\text{ almost surely}
     \end{equation*}
 and the assertion follows by \eqref{umord}.
\end{proof}
\begin{lemma}\label{incr}
Let $(L(t))_{t\in\R}$ be a two-sided L\'evy process and let $(Y_t)_{t\in\R} := (L(\frac{e^{\delta t}-1}{e^{\delta} - 1}))_{t \in \mathbb R}$ be the time-changed L\'evy process.
Then for fixed $T\in\R$ the increment process $(K_{t;T})_{t\in\R}:=(Y_{t+T}-Y_T)_{t\in\R}$ fulfills \eqref{LeK}, i.e.\ in terms of the L\'evy exponent we have
\begin{equation*}
  \Psi_{t_1,\ldots, t_k;T}^K(\theta_1,\ldots,\theta_k) = e^{\delta T}
  \Psi_{t_1,\ldots, t_k}^Y(\theta_1,\ldots,\theta_k)
\end{equation*}
for all $k\in\nat$ and $t_1,\ldots,t_k,\theta_1,\ldots,\theta_k\in\R$.
\end{lemma}
\begin{proof}
Since $(L(t))_{t\in \mathbb R}$ is a L\'evy process, for $t,T\in\R$ we have
$$K_{t;T}=  L\Big(\frac{e^{\delta( t +  T)}-1}{e^{
 \delta} - 1}\Big) -  L\Big(\frac{e^{\delta T}-1}{e^{ \delta} - 1}\Big) \stackrel{\rm d}{=} L\Big(e^{\delta T}\, \frac{e^{\delta t}-1}{e^{ \delta} - 1}\Big)$$
which in terms of the L\'evy exponents gives
$$\Psi_{t;T}^{K}(\theta) = \Psi_{ e^{\delta T}\frac{e^{\delta t}-1}{e^{ \delta} - 1} }^{L}(\theta)  =  e^{\delta T} \, \frac{e^{\delta t}-1}{e^{ \delta} - 1}  \Psi_{1}^{L}(\theta)=e^{\delta T}\Psi_{t }^{Y}(\theta)$$
for all $\theta\in\R$ which shows that \eqref{LeK} holds for $k=1$. Now the remaining case $k\geq2$ follows as in the proof of Lemma \ref{incre}.
\end{proof}
\begin{lemma}\label{intconv}
Let $(L(t))_{t\in\R}$ be a two-sided L\'evy process and let $(Y_t)_{t\in\R} := (L(\frac{e^{\delta t}-1}{e^{\delta} - 1}))_{t \in \mathbb R}$ be the time-changed L\'evy process. Then each of the following conditions is sufficient for the almost sure convergence of
$$\int_a^b e^{t(\alpha-\delta/2)}\,dY_t\quad\text{ as $a\downarrow -\infty$ for any }b\in\R.$$
\begin{itemize}
\item[\it (a)] $\delta>0$ and $\alpha>\delta/2$.
\item[\it (b)] $\delta<0$, $\alpha>-\delta/2$ and
\begin{equation}\label{momentcond}
\sup_{\frac{1}{\delta} \log 2\leq s\leq 0} \E \left[\Big|\int_{s}^0 e^{t(\alpha -
\delta /2)}\,dY_t\Big|^{\gamma}\right] < \infty\quad\text{ for some }\gamma>\frac{-\delta}{\alpha+\delta/2}.
\end{equation}
\end{itemize}
\end{lemma}
\begin{proof}
Let $b\geq a_n\downarrow -\infty$ be an arbitrary sequence and $\delta\not=0$. Choose $k_0 \in \mathbb N$ such that $-\frac{1}{|\delta|} \log(k_0) \le b $ and $k(n) \in \mathbb N_0$ such that 
$$-\frac{1}{|\delta|} \log\big(2^{k(n)+1}k_0\big) < a_n \leq -\frac{1}{|\delta|} \log\big(2^{k(n)}k_0\big).$$
Setting $\gamma_n = - \frac{1}{|\delta|} \log (2^nk_0)$ we decompose 
\begin{align*} 
\int_{a_n}^b e^{t(\alpha - \delta /2)} dY_t
&= \int_{\gamma_0}^{b}
e^{t(\alpha - \delta /2)} dY_t  + \int_{\gamma_{k(n)}}^{\gamma_0} e^{t(\alpha - \delta /2)} dY_t  + \int_{a_n}^{\gamma_{k(n)}} e^{t(\alpha - \delta /2)} dY_t \\ 
&=: A +B_n +C_n ,
\end{align*}
where $A,B_n,C_n$ are independent and $A$ is a fixed random variable. Now observe that 
\begin{align*}
 B_n= \sum_{k=0}^{k(n)-1} \int_{\gamma_{k+1}}^{\gamma_k} e^{t(\alpha - \delta/2)} dY_t =: \sum_{k=0}^{k(n)-1} Z_k
\end{align*}
is a sum of independent random variables $(Z_k)_{k\in\nat}$. Note that $(Z_k)_{k\in\nat}$ is a sequence of infinitely divisible random variables by Theorem 9.1 in \cite{Sat1}, since for any $T\in\R$ the process $(\int_T^{T+s}e^{t(\alpha-\delta/2)} dY_t)_{s\geq0}$ is additive as in Lemma \ref{driving}. We will now distinguish between the two cases $\delta> 0$ and $\delta<0 $.

(i) In case $\delta > 0$ for any $k \in \mathbb N_0$ we get by a change of variables
$$Z_k=\int_{\gamma_{k+1}}^{\gamma_k} e^{t (\alpha - \delta /2 )} dY_t\stackrel{d}{=} \int_{\gamma_{k+2}}^{\gamma_{k+1}} e^{(t+ \frac{1}{\delta} \log 2)  (\alpha - \delta /2 )} dK_{t;\frac{1}{\delta} \log 2} $$
and by Lemma \ref{incr} we obtain that 
\begin{align*}
 \big(K_{t;\frac{1}{\delta} \log 2}\big)_{t \in \mathbb R} \stackrel{d}{=} \big(Y_{t}^{(1)} + Y_{t}^{(2)}\big)_{t \in\R}, 
\end{align*}
where $(Y_{t}^{(1)})_{t\in \mathbb R} \stackrel{d}{=}(Y_{t}^{(2)})_{t\in \mathbb R} $ are independent copies of $(Y_{t})_{t\in \mathbb R}$. It follows that
$$Z_k\stackrel{d}{=} 2^{(\alpha - \delta/2) / \delta} \big( Z_{k+1}^{(1)} + Z_{k+1}^{(2)} \big),$$
where $Z_{k+1}^{(1)}$, $Z_{k+1}^{(2)}$ are i.i.d.\ and distributed as $Z_k$.
Let $\Psi_{k}$ be the L\'evy exponent of $Z_k$ then we obtain $ \Psi_{k}(\theta) = 2  \, \Psi_{k+1}(2^{(\alpha - \delta/2) / \delta} \theta)$ for any $\theta \in \mathbb R$ and $k\in\nat_0$.
Inductively, for the L\'evy exponent $\Psi_{B_n}$ of $B_n$ we get  
\begin{align*}
 \Psi_{B_n}(\theta) &= \sum_{k=0}^{k(n)-1} \Psi_k (\theta) = \sum_{k=0}^{k(n)-1} 2^{-k} \Psi_0(2^{-k (\alpha - \delta /2)/\delta} \theta).
\end{align*}
Since in (a) we assume $\alpha > \delta/2$, we get $\Psi_0(2^{-k (\alpha - \delta /2)/\delta} \theta)\to\Psi_0(0)=0$ as $k\to\infty$ and hence
$\Psi_{B_n}(\theta)\to g(\theta)$ for some function $g:\mathbb R \to \mathbb C$ with $g(0)=0$. By continuity of 
$\Psi_0$ at $0$, for any $\varepsilon > 0$ we can choose $\eta > 0$ such that 
\begin{align*}
 |\Psi_{0}\big( 2^{-k (\alpha - \delta /2)/\delta} \theta \big) | \leq \frac\varepsilon4 \quad\text{ for all } k \in \mathbb N_0 \text{ and }|\theta|< \eta.
\end{align*}
For any $\theta \in \mathbb R$ with $|\theta| < \eta$ we can further choose $n \in \mathbb N$ such that $|g(\theta) - \Psi_{B_n}(\theta)|\le \frac\varepsilon2$ and hence 
\begin{align*}
 |g(\theta)| &\le |g(\theta) - \Psi_{B_n}(\theta)| + | \Psi_{B_n}(\theta)| \\ 
 &\le \frac\varepsilon2 + \sum_{k=0}^{k(n)-1} 2^{-k} |\Psi_0(2^{-k (\alpha - \delta /2)/\delta} \theta)|\leq \frac\varepsilon2 + \frac\varepsilon4 \cdot \sum_{k=0}^{\infty} 2^{-k} = \varepsilon
\end{align*}
for all $|\theta|< \eta$ which shows that $g$ is continuous at $0$. By L\'evy's continuity theorem it follows that $(B_n)_{n\in \mathbb N}$ converges in distribution. We now turn to
\begin{align*}
 C_n &= \int_{a_n}^{\gamma_{k(n)}} e^{t(\alpha - \delta/2)} dY_t = \int_{a_n - \gamma_{k(n)}}^{0} e^{(t+\gamma_{k(n)})(\alpha-\delta/2)}dK_{t;\gamma_{k(n)}} \\ 
 &= \big( 2^{k(n)}k_0 \big)^{- \frac{\alpha - \delta/2}{\delta}} \cdot \int_{a_n - \gamma_{k(n)}}^0 e^{t(\alpha- \delta/2)} dK_{t;\gamma_{k(n)}}=: b_n\cdot W_n. 
\end{align*}
Then $b_n \to 0$ and the following calculation shows that $W_n \to0$ in probability. For fixed $n\in \mathbb N$ let 
$a_n - \gamma_{k(n)} = s_{0}^{(m)} = t_{0}^{(m)} \le s_{1}^{(m)}< t_{1}^{(m)} \le \ldots \le s_{m}^{(m)} < t_{m}^{(m)} = s_{m+1}^{(m)}=0$ be a partition with $\max_{k=1,\ldots,m} \big( t_{k}^{(m)} - t_{k-1}^{(m)} \big) \to 0$ as 
$m \to \infty$. Then we get almost surely
\begin{align*}
 W_n & = \lim_{m \to \infty} \sum_{j=0}^{m}e^{t_j^{(m)} (\alpha-\delta/2)} \big( K_{s_{j+1}^{(m)};\gamma_{k(n)}}- K_{s_{j}^{(m)};\gamma_{k(n)}} \big)\\
 & = \lim_{m \to \infty} \sum_{j=0}^{m}e^{t_j^{(m)} (\alpha-\delta/2)} K_{s_{j+1}^{(m)}-s_{j}^{(m)} ;s_{j}^{(m)}+\gamma_{k(n)}}.
\end{align*}
Hence for the L\'evy exponent $\Psi_{W_n}$ of $W_n$ we obtain by Lemma \ref{incr} for any $\theta \in \mathbb R$ 
\begin{align*}
\Psi_{W_n}(\theta) &= \lim_{m \to \infty} \sum_{j=0}^{m} \Psi_{s_{j+1}^{(m)}-s_{j}^{(m)};s_{j}^{(m)}+ \gamma_{k(n)}}^{K}\big( e^{t_j^{(m)} (\alpha - \delta/2) }\theta \big) \\ 
&= \lim_{m \to \infty} \sum_{j=0}^{m} e^{\delta (s_{j}^{(m)}+ \gamma_{k(n)})}  \Psi_{s_{j+1}^{(m)}-s_{j}^{(m)}}^{Y}\big( e^{t_j^{(m)} (\alpha - \delta/2) }\theta \big) \\ 
& = \big( 2^{k(n)} k_0\big)^{-1} \lim_{m \to \infty} \sum_{j=0}^{m} \Psi_{s_{j+1}^{(m)}-s_{j}^{(m)};s_{j}^{(m)}}^{K}\big( e^{t_j^{(m)} (\alpha - \delta/2) }\theta \big) \\ 
&=\big( 2^{k(n)} k_0\big)^{-1} \Psi_{V_n}(\theta),
\end{align*}
where $V_n= \int_{a_n-\gamma_{k(n)}}^{0}e^{t(\alpha-\delta/2)}dY_t$. Now for every subsequence $n' \to \infty$ there exists a further subsequence $n'' \to \infty$ with $a_{n''}- \gamma_{k(n'')} \to a \in [-\frac1\delta \log 2, \, 0 \, ]$ and hence $V_{n''} \to \int_{a}^{0} e^{t(\alpha -\delta/2)} dY_t$ in probability by stochastic continuity. Hence $\Psi_{W_{n''}}(\theta) \to 0$ for all $\theta \in \mathbb R^d$. This shows $W_n \to 0$ in probability and hence $C_n \to 0$ in probability. 

(ii) In case $\delta < 0$ for any $k \in \mathbb N_0$ we get by a change of variables
\begin{align*}
 Z_k = \int_{\gamma_{k+1}}^{\gamma_k} e^{t (\alpha - \delta/2)}dY_t  
 & =  \int_{\frac1{\delta}\log 2}^{0} e^{(t + \gamma_k) (\alpha - \delta/2)}dK_{t;\gamma_k}\\
 & =(2^kk_0)^{(\alpha-\delta/2)/\delta} \int_{\frac1{\delta}\log 2}^{0} e^{t (\alpha - \delta/2)}dK_{t;\gamma_k}
\end{align*}
and by Lemma \ref{incr} we obtain 
\begin{equation}\label{Kdec}
 \big(K_{t;\gamma_k}\big)_{t \in \mathbb R} \stackrel{d}{=} \Big(\sum_{\ell=1}^{2^kk_0}Y_{t}^{(\ell)}\Big)_{t \in\R}, 
\end{equation}
where $(Y_{t}^{(\ell)})_{t\in \mathbb R}$, $\ell\in\nat$,  are independent copies of $(Y_{t})_{t\in \mathbb R}$. It follows that
\begin{align*}
 B_n \stackrel{d}{=} k_0^{(\alpha-\delta/2)/\delta}\sum_{k=0}^{k(n)-1} 2^{k(\alpha -\delta/2)/ \delta} \sum_{\ell=1}^{2^{k}k_0} \int_{\frac1{\delta}\log 2}^{0} e^{t (\alpha - \delta/2)}dY_{t}^{(\ell)}.
\end{align*}
By our assumptions in (b), Lemma \ref{sconv} applied to $a=2$, $\beta=-(\alpha-\delta/2)/\delta>1$ and $b=k_0$ shows that $B_n$ converges in distribution as $ n\to \infty$, since the i.i.d.\ integrals have finite absolute moment of order $1/\beta=-\delta/(\alpha-\delta/2)<-\delta/(\alpha+\delta/2)$ by \eqref{momentcond}.
We will now show that $C_n\to0$ in probability. For any $\varepsilon > 0$ and $\gamma>0$ we get using \eqref{Kdec} and Markov's inequality
\begin{align*}
P\big\{ |C_n| > \varepsilon \big\}
 &= P \bigg\{ \bigg|\int_{a_n - \gamma_{k(n)}}^{0} e^{(t+ \frac{1}{\delta}\log(2^{k(n)}k_0))(\alpha - \delta/2)} dK_{t;\gamma_{k(n)}} \bigg| > \varepsilon \bigg\} \\
 & \le P \bigg\{ \sum_{\ell=1}^{2^{k(n)}k_0} \bigg|\int_{a_n - \gamma_{k(n)}}^{0} e^{t(\alpha - \delta/2)} dY_{t}^{(\ell)} \bigg| > \varepsilon ( 2^{k(n)}k_0 )^{- (\alpha - \delta/2)/\delta}\bigg\} \\
 & \le 2^{k(n)}k_0\, P \bigg\{ \bigg|\int_{a_n - \gamma_{k(n)}}^{0} e^{t(\alpha - \delta/2)} dY_{t} \bigg| > \varepsilon ( 2^{k(n)}k_0 )^{-1-(\alpha - \delta/2)/\delta}\bigg\} \\
 & \le \varepsilon^{-\gamma} (2^{k(n)}k_0)^{1+\gamma(1+(\alpha-\delta/2)/\delta)}\, \mathbb E \bigg[ \bigg|\int_{a_n - \gamma_{k(n)}}^{0} e^{t(\alpha - \delta/2)} dY_{t} \bigg|^{\gamma}\, \bigg] \\
 & \le \varepsilon^{-\gamma} (2^{k(n)}k_0)^{1+\gamma(\alpha+\delta/2)/\delta}\, \sup_{\frac1{\delta}\log2\leq s\leq0}\mathbb E \bigg[ \bigg|\int_{s}^{0} e^{t(\alpha - \delta/2)} dY_{t} \bigg|^{\gamma}\, \bigg].
  \end{align*}
The first term on the right-hand side converges to zero if $1+\gamma\frac{\alpha+\delta/2}{\delta}< 0$, or equivalently if $\gamma > \frac{-\delta}{\alpha+\delta/2}$ in which case the second term is bounded by \eqref{momentcond}.

Alltogether, in both cases (i) and (ii) we have shown that if either condition (a) or (b) is fulfilled then $(A+B_n +C_n)_{n \in \mathbb N}$ converges in distribution, which in our situation by Corollary A2.3 in \cite{JurVer} is equivalent to the asserted almost sure convergence.
\end{proof}
\begin{theorem} \label{rirep2} 
Let $( L(t))_{t \in \mathbb R}$ be a two-sided L\'evy process and $(Y_t)_{t \in \mathbb R} = ( L( \frac{e^{\delta t}-1}{e^{ \delta} - 1} ))_{t \in \mathbb R}$ be the time-changed L\'evy process such that
 one of the conditions (a) or (b) in Lemma \ref{intconv} is fullfilled. Then the process $(X_t)_{t \ge 0}$ given by
\begin{align} \label{intrep} 
 X_t:= \int_{-\infty}^{\log t} e^{u (\alpha -\delta / 2)} dY_u 
\end{align}
is well-defined and an additive $(\alpha,\delta)$-dilatively stable process.
\end{theorem}
\begin{proof}
By Lemma \ref{intconv} the random integral in \eqref{intrep} exists as an almost sure limit and thus $X=(X_t)_{t \geq0}$ is well-defined. Similar to the proof of Lemma \ref{driving} one can show that $X$ is an additive process. It remains to show that $X$ is $(\alpha,\delta)$-dilatively stable. For $0<s<t$ and $T>0$ given a sequence of partitions
$ \log s = s_0^{(n)}= t_0^{(n)} \leq s_1^{(n)} < t_1^{(n)} \leq \ldots \leq s_n^{(n)} < t_n^{(n)} = s_{n+1}^{(n)} = \log t$ with 
$\max_{k=1,\ldots,n} \big( t_{k}^{(n)}-t_{k-1}^{(n)} \big) \to 0$ we have almost surely 
\begin{align*}
 X_{tT}-X_{sT} &= \int_{\log(sT)}^{\log(tT)} e^{u(\alpha - \delta/2)} dY_u = T^{\alpha - \delta/2} \int_{\log s}^{\log t} e^{u(\alpha - \delta/2) d(Y_{u+\log T}-Y_{\log T})}\\
 & = T^{\alpha - \delta/2} \lim_{n \to \infty} \sum_{j=0}^{n} e^{t_{j}^{(n)}(\alpha - \delta/2)} \big( K_{s_{j+1}^{(n)};\log T} - K_{s_{j}^{(n)};\log T} \big).
\end{align*}
Setting $\gamma_{j}^{(n)}:= e^{t_{j}^{(n)}(\alpha - \delta/2)}T^{(\alpha - \delta/2)} \theta$ for arbitrary $\theta \in \mathbb R$, by Lemma \ref{incr} we obtain for the Fourier transforms
\begin{equation}\label{Fou}\begin{split}
 \widehat{P}_{X_{tT} - X_{sT}}(\theta) &= \lim_{n \to \infty} \prod_{j=0}^{n} \exp \Big( \Psi_{s_{j}^{(n)},s_{j+1}^{(n)};\log T}^{K}  \big( \gamma_{j}^{(n)}, - \gamma_{j}^{(n)}  \big) \Big) \\ 
 & = \lim_{n \to \infty} \exp  \Big( \sum_{j=0}^{n} T^{\delta}  \Psi_{s_{j}^{(n)},s_{j+1}^{(n)}}^{Y}  \big( \gamma_{j}^{(n)}, - \gamma_{j}^{(n)}  \big) \Big)\\ 
  & = \lim_{n \to \infty} \Big( \exp  \Big( \sum_{j=0}^{n}   \Psi_{s_{j}^{(n)},s_{j+1}^{(n)}}^{Y}  \big( \gamma_{j}^{(n)}, - \gamma_{j}^{(n)}  \big) \Big) \Big)^{T^{\delta}}\\ 
  &= \Big( \widehat{P}_{X_{t} - X_{s}}(T^{\alpha - \delta / 2}\theta) \Big)^{T^{\delta}}.
\end{split}\end{equation}
By stochastic continuity, for $s \downarrow 0$ it follows that
\begin{align*}
 \widehat{P}_{X_{tT}} = \Big( \widehat{P}_{X_{t}}(T^{\alpha - \delta / 2}\theta) \Big)^{T^{\delta}} \qquad \mbox{ for all } t,T>0.
\end{align*}
Choosing $T=n^{1/\delta}$ for $n\in \mathbb N$ this shows that $(X_t)_{t \geq 0}$ is infinitely divisible and in terms 
of the L\'evy exponent fullfills 
\begin{equation} \label{LevExpRel}
 \Psi_{tT}^{X}(\theta)= T^{\delta} \, \Psi_{t}^{X}\big( T^{\alpha - \delta /2} \theta \big) \quad \text{ for all } t\ge 0,\,T>0\text{ and }\theta \in \mathbb R.
\end{equation}
It remains to show that this scaling relation holds for the finite-dimensional distributions, i.e. that \eqref{DSdefi} holds. Again
it suffices to show the case $k=2$, the general case follows analogously. For $k=2$ we get for any $0 \le t_1 < t_2$, \ $\theta_1,\theta_2 \in \mathbb R$, \ $T>0$ by 
independence of the increments and \eqref{Fou}, \eqref{LevExpRel} 
\begin{align*}
 &\exp \Big( \Psi_{t_1T,t_2T}^{X} \big( \theta_1, \theta_2 \big)\Big) = \E \Big[ \exp \big( i \theta_1 X_{t_1T} +i \theta_2 X_{t_2T}  \big) \Big] \\ 
 &=  \E \Big[ \exp \big( i (\theta_1 + \theta_2) X_{t_1T}\big)\Big] \> \E \Big[ i \theta_2 (X_{t_2T}-X_{t_1T})  \big) \Big] \\ 
 &= \exp\big( \Psi_{t_1T}^{X} (\theta_1+\theta_2) \big) \> \widehat P _{X_{t_2T} - X_{t_1T}}(\theta_2) \\ 
 &= \exp\big( T^{\delta} \Psi_{t_1}^{X} ( T^{\alpha - \delta/2}  (\theta_1+\theta_2)) \big) \> \big(\widehat P _{X_{t_2} - X_{t_1}}(T^{\alpha - \delta/2}\theta_2) \big)^{T^{\delta}} \\
 &= \E \Big[ \exp \big(i T^{\alpha - \delta/2} (\theta_1 + \theta_2) X_{t_1} +i T^{\alpha - \delta/2} \theta_2 (X_{t_2}-X_{t_1})  \big) \Big]^{T^\delta} \\ 
 &= \E \Big[ \exp \big(i T^{\alpha - \delta/2} \theta_1 X_{t_1} +i T^{\alpha - \delta/2} \theta_2 X_{t_2}\big) \Big]^{T^\delta} \\ 
 &= \exp \Big( T ^{\delta } \Psi_{t_1,t_2}^{X} \big( T^{\alpha - \delta/2}\theta_1,T^{\alpha - \delta/2} \theta_2 \big)\Big)
\end{align*}
concluding the proof.
\end{proof}
\begin{remark}
If $(X_t)_{t\geq0}$ is an additive $(\alpha,\delta)$-dilatively stable process then by Theorem \ref{rirep} we know that
\begin{equation}\label{x1}
X_1=\int_{-\infty}^0e^{u(\alpha-\delta/2)}dY_u\quad\text{ almost surely.}
\end{equation}
In case $\delta<0$ and $\alpha>-\delta/2$ we can decompose the integral as in part (ii) of the proof of Lemma \ref{intconv} into 
\begin{equation}\label{x1dec}
X_1=\sum_{k=0}^\infty2^{k(\alpha-\delta/2)/\delta}\sum_{\ell=1}^{2^k}\int_{\frac1{\delta}\log 2}^0e^{t(\alpha-\delta/2)}dY_t^{(\ell)},
\end{equation}
where $(Y_t^{(\ell)})_{t\in\R}$, $\ell\in\nat$, are i.i.d.\ copies of $(Y_t)_{t\in\R}$. In particular, by \eqref{x1} the series in \eqref{x1dec} converges almost surely. If we assume a bit more, namely that the series in \eqref{x1dec} converges absolutely almost surely, then Lemma \ref{sconv} applied to $a=2$, $b=1$ and $\beta=-(\alpha-\delta/2)/\delta>1$ shows that the moment condition
$$ \E \left[\Big|\int_{\frac1{\delta}\log 2}^0 e^{t(\alpha -
\delta /2)}\,dY_t\Big|^{\frac{-\delta}{\alpha-\delta/2}}\right] < \infty$$
of order $1/\beta=-\delta/(\alpha-\delta/2)<1$ necessarily has to be fulfilled. Since we have to assure the almost sure convergence of the integral in \eqref{x1} for arbitrary sequences decreasing to $-\infty$ in Theorem \ref{rirep2}, we asserted the stronger moment condition \eqref{momentcond} of order $\gamma>\frac{-\delta}{\alpha+\delta/2}>1/\beta$ which can get arbirary large for $\alpha\downarrow-\delta/2$. We were not able to derive a precise moment condition which is necessary and sufficient for the existence of the integral in \eqref{x1} in an almost sure sense.
\end{remark}

\section{Translatively stable processes}

In this section we restate Igl\'oi's \cite{Igloi} notion of translative stability, a generalization of stationarity for stochastic processes. Similar to stationary processes, a Lamperti-type transformation provides a close connection to  dilatively stable processes already laid out in \cite{Igloi}. Specifying this connection to the subclass of additive dilatively stable processes, we can relate our results from Section 2 to an integral representation for certain translatively stable processes of Ornstein-Uhlenbeck type.
\begin{defi}\label{transstable}
An infinitely divisible process $(V_t)_{t \in \mathbb R}$ is called {\it $\delta$-translatively stable} if for some $\delta \in \mathbb R$ in terms of the L\'evy exponent we have
 $$\Psi^{V}_{t_1+T, \ldots, t_k+T}(\theta_1, \ldots, \theta_k) = e^{\delta T} \Psi^{V}_{t_1, \ldots, t_k}(\theta_1,\ldots,\theta_k)$$
for all $T \in \R,\, t_1 , \ldots, t_k \in \R,\, \theta_1, \ldots, \theta_k \in \R$ .
\end{defi}
Note that for $\delta=0$ this definition coincides with stationarity. There appears a related scaling relation in the literature called $\delta${\it -time stability} by Kopp and Molchanov \cite{KM}. The definition of this scaling relation goes back to Mansuy's \cite{Man} concept of {\it infinite divisibility with respect to time} and was further investigated in \cite{EO,HO}. We will first state these concepts in our context of characteristic functions to compare them to translatively stable processes.

 A real-valued process $(D_t)_{t \geq 0}$, is said to be infinitely divisible with respect to time (IDT) if for any $n\in\nat$ we have
\begin{equation}\label{idt}
  \Psi^{D}_{nt_1, \ldots, nt_k}(\theta_1,\ldots,\theta_k) = n \cdot \Psi^{D}_{t_1,\ldots,t_k}(\theta_1,\ldots,\theta_k) 
\end{equation}
for all $t_1 , \ldots, t_k\geq0$ and $\theta_1, \ldots, \theta_k \in \R$ . A stochastically continuous process $(Z_t)_{t\geq0}$ is called $\delta$-time stable for some $\delta\not=0$ if for any $n\in\nat$ we have
 \begin{equation}\label{ts}
 \Psi^{Z}_{n^{1 / \delta}t_1, \ldots, n^{1/\delta}t_k}(\theta_1,\ldots,\theta_k) = n \cdot \Psi^{Z}_{t_1,\ldots,t_k}(\theta_1,\ldots,\theta_k)
 \end{equation}
for all $t_1 , \ldots, t_k\geq0$ and $\theta_1, \ldots, \theta_k \in \R$. As a direct consequence from \eqref{idt}, respectively \eqref{ts} we immediately get $D_0=0$ and $Z_0=0$ almost surely.

We will now show that these concepts are closely related to translative stability and thus examples of IDT, respectively $\delta$-time stable processes given in \cite{EO,HO,KM,Man} may also serve as examples of $\delta$-translatively stable processes.
\hfill
\newpage
\begin{lemma} Let $\delta\not=0$.

  (a) If $(V_t)_{t\in\R}$ is stochastically continuous and $\delta$-translatively stable with $V_t\to0$ in probability as $t\downarrow -\infty$ then $(Z_t:= V_{\log t})_{t\geq0}$ is $\delta$-time-stable. Conversely, if $(Z_t)_{t\geq0}$ is a $\delta$-time-stable process then $(V_t:=Z_{e^t})_{t\in\R}$ is $\delta$-translatively stable.
  
  (b) If $(V_t)_{t\in\R}$ is $\delta$-translatively stable then $(D_t:=V_{1 / \delta \log t})_{t\geq0}$ is IDT. Conversely, if $(D_t)_{t\geq0}$ is IDT and all its finite-dimensional distributions are weakly right-continuous then $(V_t:=D_{e^{\delta t}})_{t\in\R}$ is $\delta$-translatively stable.
\end{lemma}
\begin{proof}
(a) For $n\in\nat$ let $T= \frac{1}{\delta} \log t$ in Definition \ref{transstable} and let $V_{-\infty}:=0$ then for any $t_1,\ldots,t_k\geq0$ the L\'evy exponent of $(Z_t= V_{\log t})_{t \geq 0}$ fulfills
$$\Psi^Z_{n^{1/\delta}t_1,\ldots,n^{1/\delta}t_k}=\Psi^V_{\log t_1+\frac1{\delta}\log n,\ldots,\log t_k+\frac1{\delta}\log n}=e^{\log n}\Psi^V_{\log t_1,\ldots,\log t_k}=n\cdot\Psi^Z_{t_1,\ldots,t_k}$$
showing that $(Z_t)_{t\geq0}$ is $\delta$-time stable. For the converse relation we observe that for $n,m\in\nat$ and $s_i=m^{-1/\delta}t_i$
$$\Psi^Z_{(n/m)^{1/\delta}t_1,\ldots,(n/m)^{1/\delta}t_k}=n\cdot\Psi^Z_{s_1,\ldots,s_k}=\frac{n}{m}\cdot\Psi^Z_{m^{1/\delta}s_1,\ldots,m^{1/\delta}s_k}=\frac{n}{m}\cdot\Psi^Z_{t_1,\ldots,t_k}.$$
Since $(Z_t)_{t\geq0}$ is stochastically continuous, its finite-dimensional distributions are weakly right-continuous and thus for any $S>0$ and $t_1,\ldots,t_k\geq0$ we get
$$\Psi^Z_{S^{1/\delta}t_1,\ldots,S^{1/\delta}t_k}=S\cdot\Psi^Z_{t_1,\ldots,t_k}.$$
Rewriting $S=e^{\delta T}$ for $T\in\R$ this shows that $(V_t:=Z_{e^t})_{t\in\R}$ is $\delta$-translatively stable.

(b) Setting $V_{-\infty}=0=V_\infty$ we can show that $(D_t:=V_{1 / \delta \log t})_{t\geq0}$ is IDT similar to part (a). For the converse relation our assumption on weak right-continuity guarantees that we can proceed as in part (a) to show that $(V_t:=D_{e^{\delta t}})_{t\in\R}$ is $\delta$-translatively stable.
\end{proof}
As mentioned above, a Lamperti-type transformation connects the class of dilatively stable and translatively stable processes as follows. In the special case of a convolution exponent $\delta=0$, the classical Lamperti transform \cite{Lam} is known to build a one-to-one correspondence between self-similar processes and stationary processes on the real line. We use the following generalization of the Lamperti transform due to Igl\'oi \cite{Igloi}.
\begin{defi}
Let $(X_t)_{t>0}$ be a real-valued stochastic process and $\alpha, \delta\in\R$. We call the stochastic process
 \begin{equation*}
 V=(V_t:= e^{-(\alpha- \delta/2)t} X_{e^t})_{t \in \mathbb R}
 \end{equation*}
  the {\it Lamperti-type transform} of $(X_t)_{t>0}$. Its inverse on the path space given by
 \begin{equation*}
 X=(X_t:= t^{\alpha-\delta/2} V_{\log t})_{t > 0}
 \end{equation*}
 is called the {\it inverse Lamperti-type transform} of $(V_t)_{t\in\mathbb R}$.
 \end{defi}
\begin{prop} \label{OneToOne}
(a) If $(X_t)_{t\geq0}$ is $(\alpha,\delta)$-dilatively stable for some $\alpha,
  \delta \in \R$ then its Lamperti-type transform $(V_t= e^{-(\alpha- \delta/2)t}
  X_{e^t})_{t \in \mathbb R}$ is $\delta$-translatively stable.
  
(b) If $(V_t)_{t\in\R}$ is $\delta$-translatively stable for some $\delta \in \R$
  then with $X_0:=0$ its inverse Lamperti-type transform $(X_t= t^{\alpha-\delta/2} V_{\log t})_{t \geq 0}$ is 
  $(\alpha,\delta)$-dilatively stable for any $\alpha \in \R$.
 \end{prop}

 \begin{proof}
 The proof is a straightforward calculation using the scaling properties of translatively and dilatively stable processes.
 
 (a) For $T \in \R$, $t_1, \ldots, t_k \in \R$ and $\theta_1,\ldots,\theta_k\in\R$ we get
  \begin{align*}
    & \Psi^{V}_{t_1+T, \ldots, t_k + T}(\theta_1, \ldots, \theta_k) =\Psi^{X}_{e^{t_1+T}, \ldots, e^{t_k + T}}(e^{-(\alpha-\delta/2)(t_1+T)}\theta_1, \ldots, e^{-(\alpha-\delta/2)(t_k+T)}\theta_k) \\
    &\quad =\Psi^{X}_{e^{t_1}e^{T}, \ldots, e^{t_k}e^{T}}(e^{-(\alpha-\delta/2)(t_1+T)}\theta_1, \ldots, e^{-(\alpha-\delta/2)(t_k+T)}\theta_k) \\
    & \quad= e^{\delta T} \Psi^{X}_{e^{t_1}, \ldots, e^{t_k}}(e^{T(\alpha-\delta/2)}e^{-(\alpha-\delta/2)(t_1+T)}\theta_1, \ldots, e^{T(\alpha-\delta/2)}e^{-(\alpha-\delta/2)(t_k+T)}\theta_k)\\
    &\quad=e^{\delta T} \Psi^{X}_{e^{t_1}, \ldots, e^{t_k}}(e^{-(\alpha-\delta/2)t_1}\theta_1, \ldots, e^{-(\alpha-\delta/2)t_k}\theta_k) \\ 
    &\quad= e^{\delta T} \Psi^{V}_{t_1, \ldots, t_k}(\theta_1, \ldots, \theta_k).
  \end{align*}
  
(b) For $T>0$, $t_1, \ldots, t_k \geq0$ and $\theta_1,\ldots,\theta_k\in\R$ we get setting $V_{-\infty}:=0$
  \begin{align*}
   & \Psi^{X}_{Tt_1, \ldots, Tt_k}(\theta_1, \ldots, \theta_k) =\Psi^{V}_{\log (Tt_1), \ldots,\log (Tt_k)}((Tt_1)^{\alpha-\delta/2}\theta_1, \ldots, (Tt_k)^{\alpha-\delta/2}\theta_k) \\ 
   & \quad=\Psi^{V}_{\log T+\log t_1, \ldots,\log T+\log t_k}((Tt_1)^{\alpha-\delta/2}\theta_1, \ldots, (Tt_k)^{\alpha-\delta/2}\theta_k)\\
   & \quad=e^{\delta \log T }\Psi^{V}_{\log t_1, \ldots,\log t_k}(T^{\alpha-\delta/2}t_1^{\alpha-\delta/2}\theta_1, \ldots, T^{\alpha-\delta/2}t_k^{\alpha-\delta/2}\theta_k) \\ 
   & \quad=T^{\delta} \Psi^{X}_{t_1, \ldots, t_k}(T^{\alpha-\delta/2}\theta_1, \ldots, T^{\alpha-\delta/2}\theta_k)
  \end{align*}
  concluding the proof.
 \end{proof}
Further we can show that there is a close connection between additive dilatively stable processes and translatively stable wide-sense Ornstein-Uhlenbeck type (OU-type) processes as introduced in Maejima and Sato \cite{MaeSat2}. 
\begin{defi}
Let $(Y_t)_{t\in\R}$ be an additive process. A stochastic process $(U_t)_{t\in\R}$ is called {\it wide-sense OU-type} process with parameter $\lambda\in\R$ and background driving process $(Y_t)_{t\in\R}$ if
\begin{equation}\label{OUt}
U_t= e^{\lambda t } \left(U_0+ \int_0^t e^{- \lambda s} dY_s\right)\quad\text{ for all }t \in \R.
\end{equation}
\end{defi}
\begin{prop}\label{TraDilConnection}
(a) Let $(X_t)_{t\geq0}$ be an additive $(\alpha,\delta)$-dilatively stable process for some $\alpha,\delta\in\R$. Then its Lamperti-type transform $(V_t= e^{-(\alpha- \delta/2)t} X_{e^t})_{t \in
\mathbb R}$ is a $\delta$-translatively stable wide-sense OU-type process with parameter $\lambda=\delta/2-\alpha$ and driving process $(Y_t)_{t\in\R}$ as in Lemma \ref{driving}.

(b) For some $\alpha,\delta\in\R$ let $(V_t)_{t\in\R}$ be a $\delta$-translatively stable wide-sense OU-type process with parameter $\lambda=\delta/2-\alpha$ and driving process $(Y_t=(L( \frac{e^{\delta t}-1}{e^{
 \delta} - 1}))_{t \in \mathbb R}$, where $(L(t))_{t\in\R}$ is a two-sided L\'evy process. If  $e^{(\alpha-\delta/2)t} V_{t}\to 0$ in probability as $t\downarrow-\infty$ then the inverse Lamperti-type transform $(X_t= t^{\alpha-\delta/2} V_{\log t})_{t\geq0}$ is an additive $(\alpha,\delta)$-dilatively stable process.
\end{prop}
\begin{proof}
(a) By Proposition \ref{OneToOne}(a) the process $(V_t)_{t\in\R}$ is $\delta$-translatively stable and by Theorem \ref{rirep} we have
$$e^{\lambda t}\left(V_0+\int_0^te^{-\lambda u}\,dY_u\right)=e^{-(\alpha-\delta/2)t}\big(X_1+(X_{e^t}-X_1)=e^{-(\alpha- \delta/2)t} X_{e^t}=V_t$$
showing that $(V_t)_{t\in\R}$ is a wide-sense OU-type process as asserted.

(b) By Proposition \ref{OneToOne}(b) the process $(X_t)_{t\geq0}$ is $(\alpha,\delta)$-dilatively stable with $X_t\to 0=:X_0$ in probability as $t\downarrow 0$ by assumption. For $0<s<t$ we observe by \eqref{OUt}
$$X_t-X_s=t^{\alpha-\delta/2} V_{\log t}-s^{\alpha-\delta/2} V_{\log s}=\int_{\log s}^{\log t}e^{(\alpha-\delta/2)u}dL(\tfrac{e^{\delta u}-1}{e^\delta-1})$$
showing that $(X_t)_{t\geq0}$ has independent increments, is stochastically continuous and has c\`adl\`ag paths.
\end{proof}
Combining Proposition \ref{TraDilConnection} with the results of Section 2 we can directly state an integral representation for translatively stable wide-sense OU-type processes.
\begin{cor}
(a) For some $\alpha,\delta\in\R$ let $(V_t)_{t\in\R}$ be a $\delta$-translatively stable wide-sense OU-type process with parameter $\lambda=\delta/2-\alpha$ and driving process $(Y_t=(L( \frac{e^{\delta t}-1}{e^{
 \delta} - 1}))_{t \in \mathbb R}$, where $(L(t))_{t\in\R}$ is a two-sided L\'evy process. If  $e^{(\alpha-\delta/2)t} V_{t}\to 0$ in probability as $t\downarrow-\infty$ then $(V_t)_{t \in R}$ has the integral representation 
 \begin{equation*}
  V_t = \int_{-\infty}^t e^{(u-t)(\alpha-\delta/2)}d Y_u.
 \end{equation*}

(b) For $\delta\not=0$ let $( L(t))_{t \in \mathbb R}$ be a two-sided L\'evy process and $(Y_t)_{t \in \mathbb R} = ( L( \frac{e^{\delta t}-1}{e^{ \delta} - 1} ))_{t \in \mathbb R}$ be the time-changed L\'evy process such that for some $\alpha\in\R$ one of the conditions (a) or (b) in Lemma \ref{intconv} is fullfilled. Then the process $(V_t)_{t \in\R}$ given by
\begin{align*}
 V_t:=  \int_{-\infty}^{t} e^{(u-t) (\alpha -\delta / 2)} dY_u 
\end{align*}
is well-defined and a $\delta$-translatively stable wide-sense OU-type process with parameter $\lambda=\delta/2-\alpha$ and driving process $(Y_t)_{t\in\R}$.
\end{cor}
\begin{proof}
Part (a) follows directly from Proposition \ref{TraDilConnection}(b) and Theorem \ref{rirep}, whereas part (b) is a direct consequence of Theorem \ref{rirep2} together with Proposition \ref{TraDilConnection}(a).
\end{proof}
\noindent {\bf Acknowledgement.} We are grateful to M\'aty\'as Barczy and Gyula Pap for stimulating and helpful discussions on earlier versions of this manuscript.

\bibliographystyle{plain}

\end{document}